\documentclass[11pt,reqno]{amsproc}

\title[The van Dommelen and Shen singularity for Prandtl]{The van Dommelen and Shen singularity in the Prandtl equations}

\author[I.~Kukavica]{Igor Kukavica}
\address{Department of Mathematics, University of Southern California, Los Angeles, CA 90089}
\email{kukavica@usc.edu}

\author[V.~Vicol]{Vlad Vicol}
\address{Department of Mathematics, Princeton University, Princeton, NJ 08544}
\email{vvicol@math.princeton.edu}

\author[F.~Wang]{Fei Wang}
\address{Department of Mathematics, University of Southern California, Los Angeles, CA 90089}
\email{wang828@usc.edu}

\usepackage{fancyhdr}
\usepackage{comment}

\usepackage[margin=1.25in]{geometry}
\usepackage{amsmath, amsthm, amssymb}
\usepackage{times}
\usepackage{graphicx}

\usepackage[usenames,dvipsnames,svgnames,table]{xcolor}
\usepackage[colorlinks=true, pdfstartview=FitV, linkcolor=blue, citecolor=blue, urlcolor=blue]{hyperref}

\begin{document}

\def\intint{\int\!\!\!\!\int}
\def\OO{\mathcal O}
\def\SS{\mathbb S}
\def\RR{\mathbb R}
\def\TT{\mathbb T}
\def\ZZ{\mathbb Z}
\def\HH{\mathbb H}
\def\RSZ{\mathcal R}
\def\GG{\mathcal G}
\def\eps{\varepsilon}
\def\tt{\langle t\rangle}
\def\erf{\mathrm{Erf}}
\def\red#1{\textcolor{red}{#1}}
\def\mgt#1{\textcolor{magenta}{#1}}
\def\ff{\rho}
\def\gg{G}
\def\phiyy{\partial_{yy} \phi}
\def\tilde{\widetilde}

\newtheorem{theorem}{Theorem}[section]
\newtheorem{corollary}[theorem]{Corollary}
\newtheorem{proposition}[theorem]{Proposition}
\newtheorem{lemma}[theorem]{Lemma}

\theoremstyle{definition}
\newtheorem{definition}{Definition}[section]
\newtheorem{remark}[theorem]{Remark}

\def\theequation{\thesection.\arabic{equation}}
\numberwithin{equation}{section}

\def\dist{\mathop{\rm dist}\nolimits}    
\def\sgn{\mathop{\rm sgn\,}\nolimits}    
\def\Tr{\mathop{\rm Tr}\nolimits}    
\def\div{\mathop{\rm div}\nolimits}    
\def\supp{\mathop{\rm supp}\nolimits}    
 \def\indeq{\qquad{}\!\!\!\!}                     
\def\period{.}                           
\def\semicolon{\,;}                      

\def\nts#1{{\cor #1\cob}}
\def\cor{\color{red}}
\def\cog{\color{green}}
\def\cob{\color{black}}
\def\coe{\color{blue}}


\begin{abstract}
In 1980, van Dommelen and Shen provided a numerical simulation that predicts the spontaneous generation of a singularity in the Prandtl boundary layer equations from a smooth initial datum, for a nontrivial Euler background. In this paper we provide a proof of this numerical conjecture by rigorously establishing the finite time blowup of the boundary layer thickness.
\end{abstract}

\maketitle

\section{Introduction} 
We consider the 2D Prandtl boundary layer equations for the unknown velocity field $(u,v)=(u(t,x,y),v(t,x,y))$:
\begin{align} 
&\partial_t u - \partial_{yy} u + u \partial_x u + v \partial_y u = - \partial_x P^E \label{eq:P:1}\\
&\partial_x u + \partial_y v = 0 \label{eq:P:2}\\
&u|_{y=0} = v|_{y=0} = 0\label{eq:P:3}\\
&u|_{y \to \infty} = U^E.\label{eq:P:4}
\end{align}
The domain we consider is $\TT \times \RR_+ = \{ (x,y) \in \TT \times \RR \colon y \geq 0\}$, with corresponding periodic boundary conditions in $x$ for all functions.
The function $U^E = U^E(t,x)$ is the trace at $y=0$ of the tangential component of the underlying Euler velocity field $(u^E,v^E)=(u^E(t,x,y),v^E(t,x,y))$, and  $P^E=P^E(t,x)$ is the trace at $y=0$ of the Euler pressure $p^E=p^E(t,x,y)$. They obey the Bernoulli equation
\begin{align} 
\partial_t U^E + U^E \partial_x U^E = -  \partial_x P^E
\label{eq:Bernoulli}
\end{align}
for $x \in \TT$ and $t \geq 0$, with periodic boundary conditions.

The goal of this paper is to prove the formation of finite time singularities in the Prandtl boundary layer equations when the underlying Euler flow is not trivial, i.e., when $U^E \neq 0$. For this purpose,  we consider the Euler trace 
\begin{align} 
U^E &= \kappa \sin x \label{eq:U:E}\\
-\partial_x P^E &= \frac{\kappa^2}{2} \sin(2x)\label{eq:P:E}
\end{align}
proposed by van Dommelen and Shen in~\cite{VanDommelenShen80},
where $\kappa \neq 0$ is a fixed parameter.
These are stationary solutions of the Bernoulli equation \eqref{eq:Bernoulli}. Moreover, the functions $U^E$ and $P^E$ above arise as traces at $y=0$ of the stationary 2D Euler solution
\begin{align*}
u^E(x,y) &= \kappa \sin x \cos y\\
v^E(x,y) &= - \kappa \cos x \sin y\\
p^E(x,y) &=  - \frac{\kappa^2}{4} \left( \cos 2x + \cos 2y \right).
\end{align*}
It is clear that the function $(u^E,v^E)$ described above is divergence free, obeys the boundary condition $v(x,0)=0$, and yields a stationary solution of the Euler equations in $\TT \times \RR_+$. 

\begin{remark}[\bf Numerical blowup is observed]
The Lagrangian computation of van Dommelen and Shen~\cite{VanDommelenShen80} was revisited and improved by many groups in the past decades~\cite{Cowley83, CasselSmithWalker96, HongHunter03, GarganoSammartinoSciacca09,GarganoSammartinoSciaccaCassel14}. 
\begin{figure}[htb!]
\begin{center}
\includegraphics[height=4in]{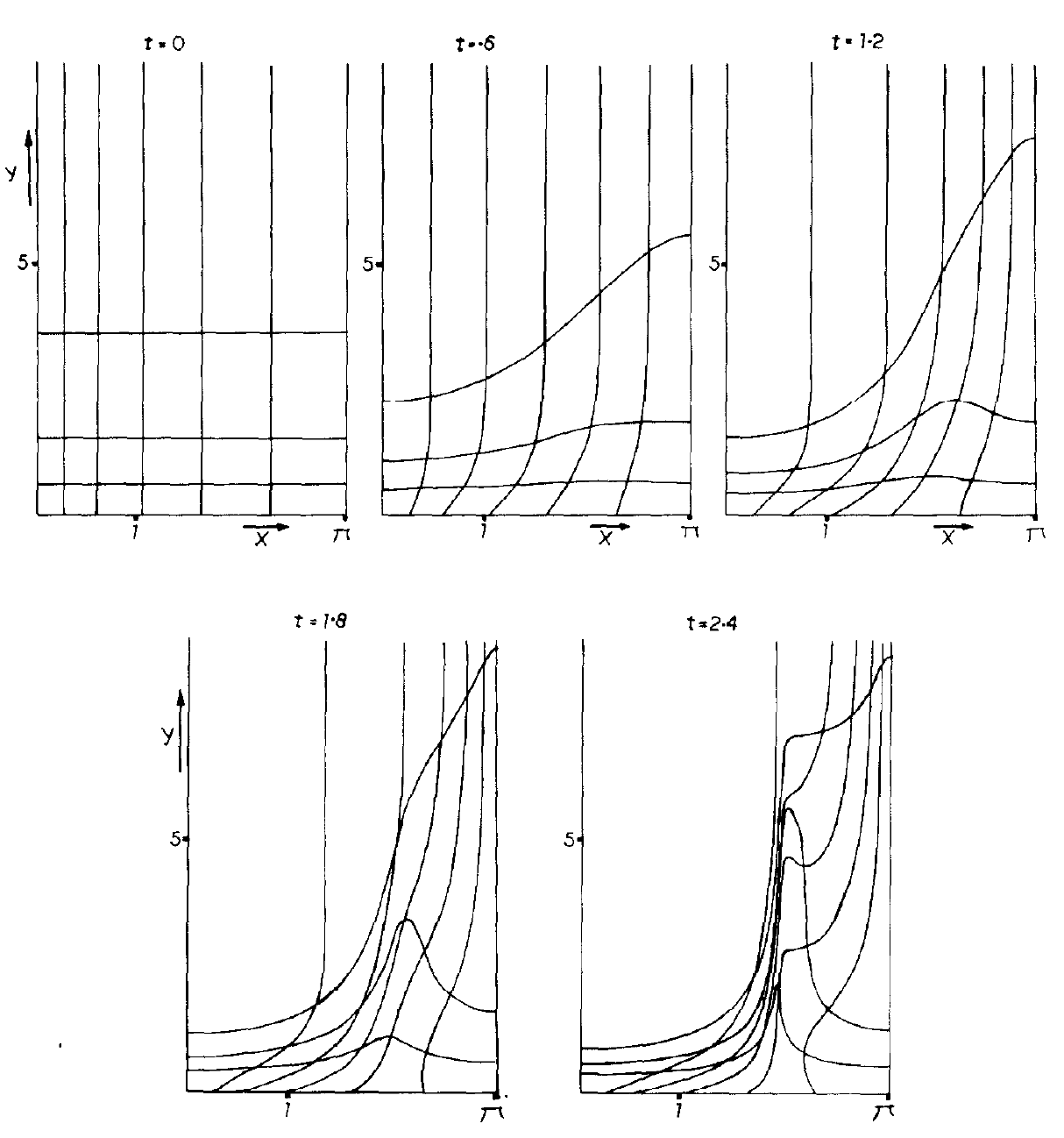}
\caption{The distortion of a typical Lagrangian grid with time. The figure is from  in~\cite[p.~127]{VanDommelenShen80}.}
\end{center}
\end{figure}
The consensus is that all the numerical experiments indicate a singularity formation in finite time from a smooth initial datum. We refer to the recent paper~\cite[Section 4.2]{CaflischGarganoSammartinoSciacca15} for a detailed discussion of the numerical singularity formation in the Prandtl system.
\end{remark}

The following is the main result of this paper:

\begin{theorem}[\bf Finite time blowup for Prandtl]
\label{thm:main}
Consider the Cauchy problem for the Prandtl equations \eqref{eq:P:1}--\eqref{eq:P:4}, with boundary conditions at $y=\infty$ matching \eqref{eq:U:E}--\eqref{eq:P:E}, with $\kappa \neq 0$. There exists a large class of initial conditions $(u_0,v_0)$ which are real-analytic in $x$ and $y$, such that the unique real-analytic solution $(u,v)$ to \eqref{eq:P:1}--\eqref{eq:P:4} blows up in finite time.
\end{theorem}

\begin{remark}[\bf Inviscid limit]
For a real-analytic initial datum in the Navier-Stokes, Euler, and Prandtl equations, it was shown in~\cite{SammartinoCaflisch98b}  that the inviscid limit of the Navier-Stokes equations is described to the leading order by the Euler solution outside of a boundary layer of thickness $\sqrt{\nu t}$, and by the Prandtl solution inside the boundary layer (see also~\cite{Maekawa14} for initial vorticity supported away from the boundary). Indeed, if any series expansion (in $\nu$) of the Navier-Stokes solution holds, then the leading order term near the boundary must be given by the Prandtl solution. The result in~\cite{SammartinoCaflisch98b} states that the inviscid limit holds on a time interval on which the Prandtl solution does not lose real-analyticity. Our result in Theorem~\ref{thm:main} shows that this time interval cannot be extended to be arbitrarily large, and thus the Prandtl expansion approach to the inviscid limit should  only be expected to hold on finite time intervals.
\end{remark}

\begin{remark}[\bf The case $\kappa=0$]
In the case of a trivial Euler flow $U^E$ and a trivial Euler pressure $P^E$, i.e., for $\kappa=0$ in \eqref{eq:U:E}--\eqref{eq:P:E}, the emergence of a finite time singularity for the Prandtl equations was established in~\cite{EEngquist97}. There, the initial datum is taken to have compact support in $y$ and be large in a certain nonlinear sense. It is shown that either the solution ceases to be smooth, or that the solution along with its derivatives does not decay sufficiently fast as $y \to \infty$. The proof given in~\cite{EEngquist97} does not appear to handle the case $\kappa\neq 0$ treated in this paper. Indeed, here the initial datum does not need to have compact support in $y$ and the pressure gradient is not trivial. Moreover, it is shown in~\cite[Appendix]{GarganoSammartinoSciacca09} that from the numerical point of view the structure of the singularity for $\kappa =0$ is different from the complex structure of the singularities in~\cite{VanDommelenShen80}.
\end{remark}

\begin{remark}[\bf Boundary layer separation and the displacement thickness]
The displacement thickness (cf.~\cite{Schlichting60,VanDommelenShen80,CousteixMauss07}) is defined as
\begin{align}
 \delta^*(t,x) =  \int_0^\infty \left( 1- \frac{u(t,x,y)}{U^E(t,x)} \right) dy = \int_0^\infty \left( 1- \frac{u(t,x,y)}{\kappa \sin(x)} \right) dy.
 \label{eq:displacement:def}
\end{align}
Physically, it measures the effect of the boundary layer on the inviscid flow~\cite{CousteixMauss07}.
\begin{figure}[htb!]
\begin{center}
\includegraphics[height=3in]{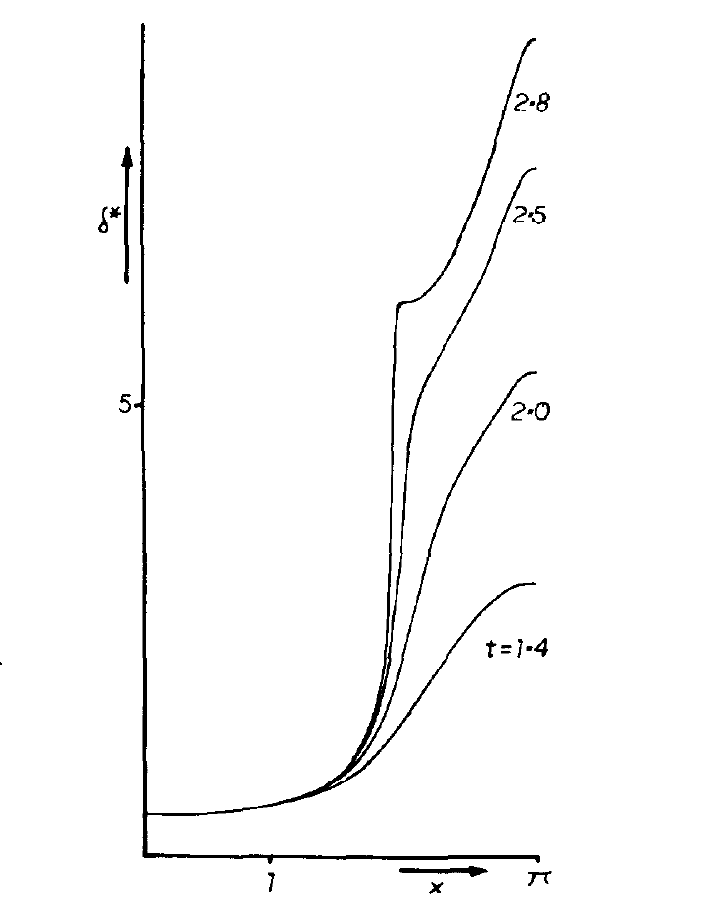}
\caption{The variation of the displacement thickness $\delta^*(t,x)$ for various time instants. The figure is from~\cite[p.~138]{VanDommelenShen80}. }
\end{center}
\end{figure}
As long as $\delta^*(t,x)$ remains bounded, the Prandtl layer remains of thickness proportional to $\sqrt{\nu}$, i.e., it remains a Prandtl layer. In turn, if the displacement thickness develops a singularity in finite time, this signals that a boundary layer separation has occurred, and after this point, the Prandtl expansion is not expected to hold anymore~(see also\cite{Grenier00,GrenierGuoNguyen14b,GrenierGuoNguyen14,GrenierGuoNguyen14c}).
For a proof of a boundary layer separation for the stationary Prandtl equations, we refer to the recent paper~\cite{DalibardMasmoudi15}. 

The proof of Theorem~\ref{thm:main} consists of showing that a Lyapunov functional $\GG(t)$ blows up in finite time. The functional $\GG$ is defined by
\begin{align*}
\GG(t) = \int_0^\infty \left(\kappa \phi(t,y) -\partial_x u(t,0,y) \right)  w(y) dy
\end{align*}
where $w(y)$ is an $L^1$ weight function and $\phi(t,y)$ is the solution of a nonhomogenous heat equation (see \eqref{eq:phi:particular} and \eqref{eq:G:def} below for details). 
We note that the Lyapunov functional was built to emulate a weighted in $y$ version of $\kappa \delta^*(t,x)|_{x=0}$. Indeed, as long as $u$ is smooth near $x=0$, we have that 
\begin{align*}
\kappa \delta^*(t,0) =  \lim_{x\to 0} \kappa \delta^*(t,x) = \int_0^\infty \left( \kappa - \partial_x u(t,0,y) \right) dy.
\end{align*}
\end{remark}

\begin{remark}[\bf More general Euler flows]
The blowup result of Theorem~\ref{thm:main} holds if the Euler flow defined \eqref{eq:U:E}--\eqref{eq:P:E} is replaced by any smooth and odd function $U^E(x)$, upon defining $P^E$ to equal $- (U^E(x))^2/2$, which is in turn even in $x$. 
\end{remark}

\begin{remark}[\bf More general classes of initial conditions]
\label{rem:datum}
We note that besides yielding the local existence of solutions (cf.~\cite{SammartinoCaflisch98a}), the analyticity of the initial datum is {\em not} required for proving Theorem~\ref{thm:main}. One may instead consider an initial datum that is merely Sobolev smooth with respect to $y$, analytic with respect to $x$, and decays sufficiently fast as $y\to \infty$ (cf.~\cite{CannoneLombardoSammartino01,KukavicaVicol13a,IgnatovaVicol15}). Alternatively Gevrey-class $7/4$ regularity in $x$ may be considered~\cite{GerardVaretMasmoudi13}, whose vorticity decays sufficiently fast with respect to $y$. On the other hand, in view of the oddness in $x$ of the boundary condition~\eqref{eq:U:E}, we cannot consider initial datum which is in the Oleinik class of monotone in $y$ flows (uniformly with respect to $x$), although the local existence holds in this class~\cite{Oleinik66,MasmoudiWong12a,AlexandreWangXuYang15}. The datum in~\cite{XuZhang15} is also not allowed in view of the boundary conditions~\eqref{eq:U:E}. The mixed analyticity near $x=0$, and monotonicity away from the $y$-axis (of different signs) may however be treated, using the local existence result in~\cite{KukavicaMasmoudiVicolWong14}.
\end{remark}

\begin{remark}[\bf Ill-posedness for the Prandtl equations]
The ill-posedness of the Prandtl equations was established at the linear level in~\cite{GerardVaretDormy10}, and at the nonlinear level in~\cite{Grenier00,GuoNguyen11,GerardVaretNguyen12}. We note that these results do not imply the finite time blowup from a given initial datum.
\end{remark}

\begin{remark}[\bf Finite time blowup for the hyrdostatic Euler equations]
Here we note certain very interesting blowup results~\cite{CaoIbrahimNakanishiTiti13,Wong15} for the hydrostatic Euler equations (which has many analogies with the Prandtl equations).
These equations are set in a finite strip $\Omega=R\times[0,h]$, and have different boundary conditions (only $v=0$ is imposed at the top and bottom boundaries). For these equations the local existence was established for convex~\cite{Brenier99,MasmoudiWong12b} or analytic~\cite{KukavicaTemamVicolZiane11} data, as well as for the combination thereof~\cite{KukavicaMasmoudiVicolWong14}. These equations are at the same time severely unstable (i.e., ill-posed in Sobolev spaces) if convexity or analyticity is absent \cite{Renardy09}. In~\cite{CaoIbrahimNakanishiTiti13} and \cite{Wong15} the finite blowup of odd solutions is established, by observing the behavior of $u_x=-w_z$. The main difference with~\cite{EEngquist97} is the presence of the pressure, which is quadratic in $u$ (cf.~\cite{CaoIbrahimNakanishiTiti13} for further details).
\end{remark}

The paper is organized as follows. In Section~\ref{sec:proof} we give the proof of Theorem~\ref{thm:main} assuming certain properties of a boundary condition lift $\phi=\phi(t,y)$ and of a weight function $w=w(y)$. Sections~\ref{sec:phi} and~\ref{sec:w} are devoted to establishing these properties of $\phi$ and $w$ respectively.

\section{Proof of Theorem~\ref{thm:main}}
\label{sec:proof}
\subsection{Local existence}
As initial datum for the Prandtl equation we consider 
\begin{align}
 u_0(x,y) = \kappa \erf\left(\frac{y}{2}\right) \sin x + \bar u_0(x,y)
 \label{eq:IC:1}
\end{align}
where
$\bar u_0(x,y)$ is a real-analytic function of $x$ and $y$ which is also {\em odd} with respect to $x$, and $\erf(z) = 2 \pi^{-1/2} \int_0^z \exp(-\bar z^2) d\bar z$ is the Gauss error function. We assume that $\bar u_0$ decays sufficiently fast (at least at an integrable rates) as $y\to \infty$, and takes the value $0$ at $y=0$. Moreover, we assume that
\begin{align*}
 a_0(y) = (-\partial_x \bar u_0)(0,y)
\end{align*}
obeys
\begin{align*}
 a_0(y) >0 \quad \mbox{for all} \quad y>0
\end{align*}
and that 
\begin{align}
 \GG_0 = \int_0^\infty a_0(y) w(y) dy \geq \bar C_{\kappa,w}
 \label{eq:G:big}
\end{align}
where the weight $w$ is as constructed in Section~\ref{sec:w}, and $\bar C_{\kappa,w} \geq 1$ is a constant that depends on $\kappa$ and on the $L^1(\RR_+)$ weight function $w$. 

For instance, we may consider 
\begin{align}
\bar u_0(x,y) = - Ay^2 \exp(-y^2) \sin x
\label{eq:IC:2}
\end{align}
where $A = A(\kappa,w)>0$ is a sufficiently large constant. This choice for $\bar u_0(x,y)$ yields that  $a_0(y)$ is a large constant multiple of $\varphi(y) = y^2 \exp(-y^2)$. The local in time existence of a unique real-analytic solution of the Prandtl system \eqref{eq:P:1}--\eqref{eq:P:E} with initial datum given by \eqref{eq:IC:1}--\eqref{eq:IC:2} follows from \cite{SammartinoCaflisch98a}.  

We note though that the real-analyticity is not needed in the blowup proof. It is only used to ensure that we have the local in time existence and uniqueness of smooth solutions. Much more general classes of initial conditions $\bar u_0$ may be considered as long as they are odd with respect to $x$, the Cauchy problem is locally well-posed (cf.~Remark~\ref{rem:datum} above), and \eqref{eq:G:big} holds.

\subsection{Restriction of Prandtl dynamics on the $y$-axis}
Consider an initial datum $u_0(x,y)$ for the Prandtl equations that is odd in $x$ (such as the one defined in \eqref{eq:IC:1}). Note that the boundary condition at $y=0$ is homogenous and thus automatically odd in $x$,  the boundary condition at $y=\infty$ given by the Euler trace in \eqref{eq:U:E} is also odd in $x$, and the derivative of the Euler pressure trace~\eqref{eq:P:E} is odd in $x$ as well. Therefore, the unique classical solution $u(t,x,y)$ of \eqref{eq:P:1}--\eqref{eq:P:E} is  also odd in $x$. Hence, as long as the solution remains smooth we have
\begin{align}
u(t,0,y) = (\partial_y u)(t,0,y) = (\partial_x^2 u)(t,0,y) =0.
\label{eq:oddness}
\end{align}
Physically, this symmetry freezes the Lagrangian paths emanating from the $y$-axis, introducing a stable stagnation point in the flow. 
As in~\cite{EEngquist97} (see also~\cite{CaoIbrahimNakanishiTiti13,Wong15}) this allows one to consider the dynamics obeyed by the tangential derivative of $u$ at $x=0$, i.e.,
\begin{align*} 
b(t,y) = -(\partial_x u)(t,x,y)|_{x=0}.
\end{align*}
As long as the solution remains smooth (so that we may take traces at $x=0$), using \eqref{eq:oddness} one derives that the equation obeyed by $b(t,y)$ is
\begin{align} 
&\partial_t b - \partial_{yy} b - b^2 + \partial_y^{-1} b \, \partial_y b= - \kappa^2 \label{eq:b:1}\\
&b|_{y=0} =0  \label{eq:b:2} \\
&b|_{y \to \infty} = - \kappa.  \label{eq:b:3}
\end{align}
In \eqref{eq:b:1} and throughout the paper, we denote the integration with respect to the vertical variable as
\begin{align*}
 \partial_y^{-1} \varphi(t,y) = \int_0^y \varphi(t,y') dy'
\end{align*}
for any function $\varphi(t,y)$ which is integrable in $y$.
In order to obtain \eqref{eq:b:1}, one applies $-\partial_x$ to \eqref{eq:P:1} and then evaluates the resulting equation on the $y$-axis. Similarly, \eqref{eq:b:3} follows upon taking a derivative with respect to $x$ of \eqref{eq:U:E} and setting $x=0$.

\subsection{A shift of the boundary conditions}
In order to homogenize the boundary condition at $y=\infty$ when $t=0$, we add to $b$ a   lift $\phi(t,y)$ defined as the solution of the nonhomogenous heat equation
\begin{align} 
&\partial_t \phi - \partial_{yy} \phi = \kappa^2 \label{eq:phi:1}\\
&\phi|_{y=0} =0 \label{eq:phi:2} \\
&\phi|_{y \to \infty} =  \kappa + \kappa^2 t \label{eq:phi:3}
\end{align}
with an initial datum that we may choose, as long as it obeys compatible boundary conditions. We consider 
\begin{align*}
\phi_0(y) =   \kappa \erf\left( \frac y 2\right),
\end{align*}
so that the solution of \eqref{eq:phi:1}--\eqref{eq:phi:3} is explicit
\begin{align}
 \phi(t,y) &=   \kappa \erf\left( \frac{y}{\sqrt{4(t+1)}} \right) \notag\\
&\qquad + \kappa^2 t  \left( \frac{y^2}{2t} \left( \erf\left( \frac{y}{\sqrt{4 t}} \right) - 1 \right) + \left( \erf\left( \frac{y}{\sqrt{4 t}} \right) + \frac{y}{\sqrt{\pi t}} \exp\left( - \frac{y^2}{4 t}\right) \right) \right).
\label{eq:phi:particular}
\end{align}
In Section~\ref{sec:phi} we prove a number of properties (such as $\phi \geq 0$ and that $\partial_y \phi \geq 0$) of the function $\phi$ defined in \eqref{eq:phi:particular}.

Letting 
\begin{align} 
a(t,y) = b(t,y) + \phi(t,y)
\label{eq:a:def}
\end{align} 
the system \eqref{eq:b:1}--\eqref{eq:b:3} becomes
\begin{align} 
&\partial_t a - \partial_{yy} a = (a - \phi)^2 - \partial_y^{-1} (a-\phi) \, \partial_y (a - \phi) \label{eq:a:1} \\
&a|_{y=0} =0 \label{eq:a:2} \\
&a|_{y \to \infty} = \kappa^2 t. \label{eq:a:3}
\end{align}
The equation \eqref{eq:a:1} is similar to the one obtained in~\cite{EEngquist97} for $\kappa = 0$, except for two additional terms on the right side: a forcing term
\begin{align} 
F(t,y) = \phi(t,y)^2 - \partial_y^{-1} \phi(t,y) \, \partial_y \phi(t,y) 
\label{eq:F:def}
\end{align}
and a linear term
\begin{align} 
L[a](t,y) = - 2 a(t,y) \phi(t,y) + \partial_y^{-1} \phi(t,y) \, \partial_y a(t,y) + \partial_y^{-1} a(t,y) \, \partial_y \phi(t,y) .
\label{eq:L:def}
\end{align}
The forcing term is explicit in view of \eqref{eq:phi:particular}, while the linear operator $L[a]$  has nice coefficients given in terms of $\phi$.
With the notation \eqref{eq:F:def}--\eqref{eq:L:def},  the evolution equation for $a$ becomes
\begin{align} 
&\partial_t a - \partial_{yy} a = a^2 -  \partial_y^{-1} a \, \partial_y a  + L [a] + F 
\label{eq:a:evo}\\
&a|_{y=0} =0 
\label{eq:a:BC:1}\\
&a|_{y \to \infty} = \kappa^2 t.
\label{eq:a:BC:2}
\end{align}
In order to prove Theorem~\ref{thm:main}, we show that the solution of \eqref{eq:a:evo}--\eqref{eq:a:BC:2} blows up in finite time from a very large class of smooth initial data $a_0$.

\subsection{Minimum principle}
The main purpose of shifting the function $b$ up by $\phi$ is so that the resulting function $a$ obeys a positivity principle.
\begin{lemma}
\label{lem:min}
Assume that $a_0 = a(0,y)$ is such that $a_0(y)>0$ for all $y \in (0,\infty)$. Consider a smooth solution $a(t,y)$ of the initial value problem associated with \eqref{eq:a:evo}--\eqref{eq:a:BC:2} and initial condition $a_0$, on a time interval $[0,T]$. Then we have that $a(t,y) \geq 0$ for all $y\geq 0$ and $t \in [0,T]$.
\end{lemma}

Before proving Lemma~\ref{lem:min}, we need to establish certain positivity properties concerning the function $\phi$.
\begin{lemma}
\label{lem:F}
Let $\phi(t,y)$ be as defined in \eqref{eq:phi:particular} with $\kappa > 0$. Then we have that 
\begin{align} 
\phi(t,y) &\geq 0 \label{eq:phi:positive}\\
\phi(t,y) &\leq C_\kappa (1+t) \label{eq:phi:bounded}\\
\partial_y \phi(t,y) &\geq 0 \label{eq:phi:increasing}\\
\partial_{yy} \phi(t,y) &\leq 0 \label{eq:phi:concave}
\end{align}
for all $t\geq 0$ and all $y\geq 0$,
where  $C_\kappa>0$ is a constant that depends only on $\kappa$. Moreover, the inequality in \eqref{eq:phi:positive} is strict for $y>0$.
\end{lemma}
The proof of Lemma~\ref{lem:F} is given in Section~\ref{sec:phi} below.

\begin{proof}[Proof of Lemma~\ref{lem:min}]
We argue by contradiction.
Since the solution is classical on $[0,T]$ and decays sufficiently fast as $y\to \infty$, in order to reach a strictly negative value in $[0,T] \times (0,\infty)$ there must exist a first time $t_0$ and an interior point $y_0 > 0$, such that 
\begin{align*}
a(t_0,y_0)&= 0\\
a(t_0,y) &\geq 0 \quad \mbox{for all} \quad y \in \RR_+\\
(\partial_t a)(t_0,y_0) &\leq 0.
\end{align*}
As is classical for the heat equation the contradiction arises by computing the time derivative of $a$ at the point $(t_0,y_0)$ and showing that it is strictly positive, contradicting  the minimality of $t_0$. In order to bound $(\partial_t a)(t_0,y_0)$ from below we use \eqref{eq:a:evo}. Since $a(t_0,\cdot)$ has a global minimum at the interior point $y_0$, we have  
\begin{align*}
(\partial_y a)(t_0,y_0) &= 0 \\
(\partial_{yy} a)(t_0,y_0) &\geq 0.
\end{align*}
Since by assumption $a(t_0,y_0)=0$, it follows that $(\partial_{yy} a + a^2 - \partial_y^{-1} a \, \partial_y a)(t_0,y_0) \geq 0$.
Moreover, since $\phi$ is a non-negative non-decreasing function, we immediately obtain from \eqref{eq:L:def} that $L[a](t_0,y_0) \geq 0$. 
We conclude the proof by showing that $F(t_0,y_0) >0$. Indeed, by \eqref{eq:phi:positive} and \eqref{eq:phi:concave} we have that 
\begin{align*}
\partial_y F = \phi \partial_y \phi - \partial_y^{-1} \phi \, \partial_{yy} \phi \geq \phi \partial_y \phi = \frac 12 \partial_y (\phi^2)
\end{align*}
and thus
\begin{align}
F(t,y) = \int_0^y \partial_y F(t,\bar y) d\bar y \geq \frac 12 \int_0^y \partial_y (\phi^2)(t,\bar y) d \bar y = \frac 12 (\phi(t,y))^2 > 0
\label{eq:F:positive}
\end{align}
whenever $y>0$, in view of Lemma~\ref{lem:F}.

In order to fully justify this argument, we apply the proof to  $\tilde a(t,y) = a(t,y) + \eps$ and show that $\tilde a(t,y)$ remains non-negative for every $\eps >0$. The latter requires the additional observation that 
\[
L[\eps](t,y) = - 2 \eps \phi(t,y) + \eps y \partial_y \phi(t,y) = \eps \int_0^y \left( \bar y \partial_{yy} \phi (t,\bar y) - \partial_y  \phi(t,\bar y) \right) d\bar y \leq 0
\]
in view of Lemma~\ref{lem:F}. This concludes the proof of the minimum principle.
\end{proof}

\subsection{Blowup of a Lyapunov functional}
Motivated by the displacement thickness (cf.~\eqref{eq:displacement:def}) we consider the evolution of the weighted average of $a(t,y)$ on $\RR_+$. For a suitable weight $w(y)$ to be defined below and a non-negative solution $a$ of \eqref{eq:a:evo}--\eqref{eq:a:BC:2}, we define the Lyapunov functional
\begin{align} 
\GG(t) = \int_0^\infty a(t,y) w(y) dy.
\label{eq:G:def}
\end{align}
Note that since $a \geq 0$ as long as $a$ remains smooth (cf.~Lemma~\ref{lem:min}) we have that
\begin{align*} 
\GG(t) \geq 0
\end{align*}
for all $t \geq 0$.
Our goal is to establish an inequality of the type 
\begin{align}
\frac{d\GG}{dt} 
\geq \frac{1}{C} \GG^2 - C (1+t) (1+\GG)
\label{eq:to:show}
\end{align}
for a constant $C \geq 1$. Choosing a suitable initial datum, we then conclude that $\GG$ blows up in finite time. The first step is to present the properties of the weight $w$ in \eqref{eq:G:def} which are needed in the proof of \eqref{eq:to:show}.

\subsubsection{Properties of the weight function $w$}
We consider a weight function $w(y)$ such that:
\begin{align*} 
& w \in W^{2,\infty}(\RR_+)\\
& w \geq 0 \\
& w|_{y=0} = w|_{y \to \infty} = 0 \\
& w \in L^1(\RR_+).
\end{align*}
The weight is given by glueing two functions $f$ and $g$, i.e.,
\begin{align*} 
w(y) = \begin{cases}
f(y), 0\leq y \leq Q,\\
g(y), y\geq Q,
\end{cases}
\end{align*}
where $Q > 0$; the function $f$ is such that
\begin{align} 
& f(0) = 0 \label{eq:f:0}\\
& f \geq 0, \mbox{ on }  [0, Q]
\label{eq:IK:f>0}\\
& f'' \leq 0, \mbox{ on }  [0,Q] 
\label{eq:IK:f:d2<0}\\
&f''(y) \geq  -c_f f(y), \mbox{ for all } y \in [0, Q], \mbox{ where } c_{f}> 0 
\label{eq:IK:f:d2:lower} \\
&y f'(y) \leq \bar c_f f(y), \mbox{ for all } y \in [0,Q], \mbox{ where } \bar c_f >0
\label{eq:IK:f:d1}
\end{align}
and the function $g$, which we extend to be defined on $[M,\infty)$ for some $M \in (0,Q)$ obeys
\begin{align} 
& \lim_{y \to \infty} g(y) = 0 = \lim_{y\to \infty} g'(y)\label{eq:g:infty}\\
& g(y)> 0, \mbox{ for all } y \geq M\\
&g'(y) <  0, \mbox{ for all } y \geq M
\label{eq:IK:g:decreasing}\\
&g''(y) > 0, \mbox{ for all } y\geq M \label{eq:g:convex}\\
& \frac{g'(y)^2}{g(y) g''(y) } \leq \beta < 1, \mbox{ for all } y \geq Q. \label{eq:IK:KEY:1}
\end{align}
Also let $\psi(z)$ be a smooth non-decreasing cutoff function such that $\psi(z)=0$ for all $z\leq 0$, $\psi(z)=1$ for all $z\geq 1$, and $|\psi'(z)|\leq 2$ for all $y \in (0,1)$. Define the cutoff function
\begin{align} 
\eta(y) = \psi\left( \frac{y-M}{Q-M} \right).
\label{eq:eta:def}
\end{align}
Note that $\eta$ vanishes identically on $[0,M ]$ and equals $1$ on $[Q,\infty)$. Its derivative localizes to $[M,Q]$ and obeys
\begin{align*} 
0 \leq \eta'(y) \leq \frac{2}{Q-M}
\end{align*}
for all $y \in (M,Q)$.
Lastly, we require that the functions $f$ and $g$ obey the compatibility conditions
\begin{align} 
& \eta(y) \frac{g'(y)^2}{f(y) g''(y)} \leq \beta < 1, \mbox{ for all } y \in (M,Q) \label{eq:IK:KEY:2}\\
&2 \eta'(y) |g'(y)| \leq \eta(y) g''(y) - f''(y), \mbox{ for all } y \in (M, Q). \label{eq:IK:KEY:3}
\end{align}
The construction of two functions $f$ and $g$ that obey the properties \eqref{eq:f:0}--\eqref{eq:IK:KEY:3} listed above is provided in Section~\ref{sec:w}. A sketch of the graph of the resulting weight $w(y)$ is given in Figure~\ref{fig:w} below.
Throughout paper, we shall denote derivatives of the functions $w,f,g$ with primes, as they are only functions of the variable $y$.

\begin{figure}[htb!]
\begin{center}
\setlength{\unitlength}{1.5in} 
\begin{picture}(4,2)
\put(0.25,0){\includegraphics[height=3in]{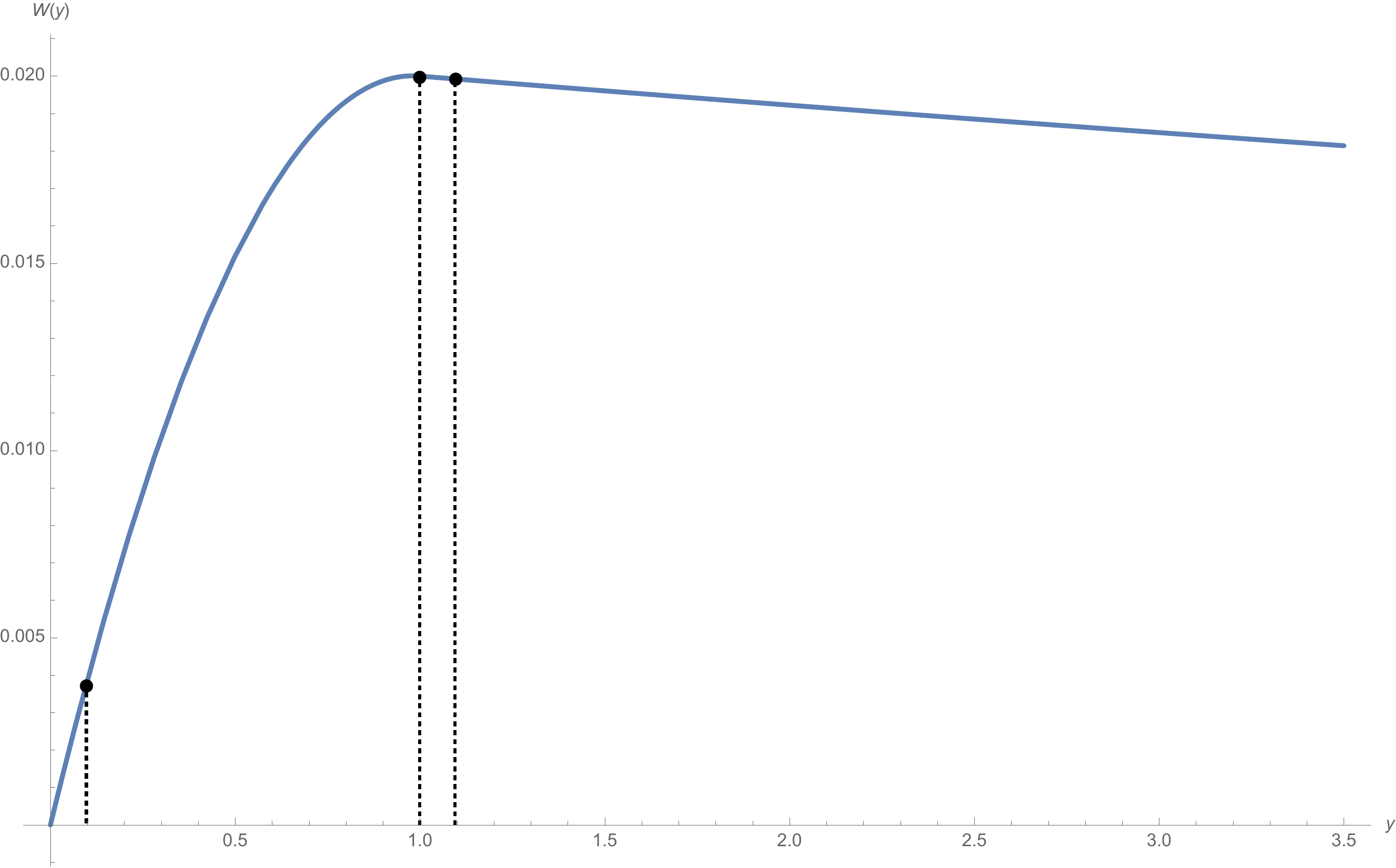}}
\put(0.55,1){\mbox{{$f$}}}
\put(1.1,0){\mbox{{$M$}}}
\put(1.3,0){\mbox{{$Q$}}}
\put(0.43,0){\mbox{{$\eps$}}}
\put(2.1,1.8){\mbox{{$g$}}}
\end{picture}
\caption{Graph of the weight function $w(y)$. The weight is a linear function on $[0,\eps]$, a quadratic function on $[\eps,Q]$, and an shifted negative power of $y$ on $[Q,\infty)$.}
\label{fig:w}
\end{center}
\end{figure}
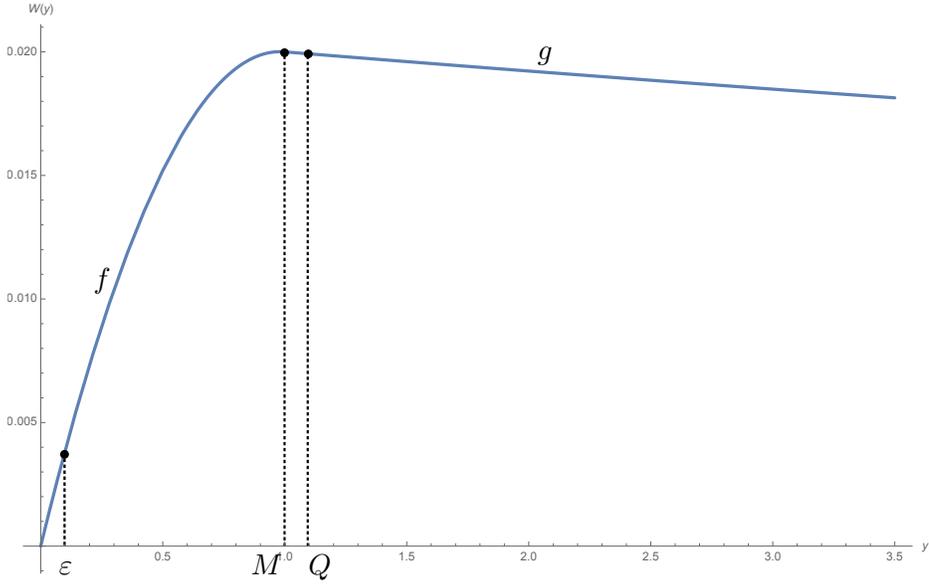

\subsubsection{Evolution of the Lyapunov functional $\GG$}
We use \eqref{eq:a:evo}--\eqref{eq:a:BC:2} and the boundary values of $w$, given by \eqref{eq:f:0} and \eqref{eq:g:infty}, to deduce
\begin{align} 
\frac{d}{dt} \GG 
&= \int_0^\infty \left( \partial_{yy} a + a^2 - \partial_y^{-1} a\,\partial_y a + L[a] + F \right) w  dy \notag \\
&= - \int_0^\infty \partial_y a w' dy + 2 \int_0^\infty a^2 w dy + \int_0^\infty a \partial_y^{-1}a\, w' dy + \int_0^\infty L[a] w dy + \int_0^\infty F w dy\notag \\
&\geq \int_0^\infty a w'' dy + 2 \int_0^\infty a^2 w dy - \frac 12 \int_0^\infty (\partial_y^{-1} a)^2  w''  dy + \int_0^\infty L[a] w dy \notag \\
&= I_1 + 2 I_2 - \frac 12 I_3 + I_4. \label{eq:IK:dt:G}
\end{align}
The above integrations by parts are justified by the fact that $a$ and $w$ are sufficiently smooth, $w$ obeys the Dirichlet boundary conditions, $a$ and $\partial_{y}^{-1} a$ vanish at $y=0$, while $\partial_y w$ vanishes as $y\to \infty$. Here we have used that by \eqref{eq:F:positive} we have $F(t,y) \geq 0$, which combined with $w(y) \geq 0$ shows that the forcing is non-negative.
We now bound each of the terms in \eqref{eq:IK:dt:G} separately.

\subsubsection{Bound for $I_1$}
To bound $I_1$ we use the convexity of $w$ on $[Q,\infty)$ (cf.~\eqref{eq:g:convex}) and the fact that \eqref{eq:IK:f:d2:lower} holds. We  deduce that 
\begin{align}
I_1 \geq \int_0^Q a f''  \geq - c_f \int_0^{Q} a f \geq -c_f \GG.
\label{eq:IK:I1}
\end{align}

\subsubsection{Bound for $I_2$}
By the Cauchy-Schwartz inequality, 
and denoting
\begin{align}
c_1= \|w\|_{L^1(\RR_+)} < \infty,
\label{eq:IK:w:L1}
\end{align}
we obtain  
\begin{align*} 
\GG = \int_0^\infty a w dy = \int_0^\infty a \sqrt{w} \sqrt{w} dy \leq \|a \sqrt{w} \|_{L^2} \|\sqrt{w} \|_{L^2}= c_1^{1/2} I_2^{1/2}
\end{align*}
and thus
\begin{align} 
I_2 \geq \frac{1}{c_1} \GG^2.\label{eq:IK:I2}
\end{align}

\subsubsection{Bound for $I_3$}
We proceed to bound the difficult term $I_3$. Let $\eta(y)$ be the cutoff function defined in \eqref{eq:eta:def}. First we have
\begin{align}
I_3 & = \int_0^\infty (\partial_y^{-1} a)^2 w''  \notag\\
& = \int_0^\infty \eta (\partial_y^{-1} a)^2 w'' + \int_0^\infty (1-\eta) (\partial_y^{-1} a)^2 f'' \notag\\
&\leq \int_0^\infty \eta (\partial_y^{-1} a)^2 g'' +  \int_M^{Q} \eta (\partial_y^{-1} a)^2 (f''- g'') + \int_M^Q (1-\eta) (\partial_y^{-1} a)^2 f''  \notag \\
&= J + \int_M^{Q}  (\partial_y^{-1} a)^2 (f'' - \eta g'')  
\label{eq:IK:I3:1}
\end{align}
where we have used \eqref{eq:IK:f:d2<0} in the second to last step to bound 
\[
\int_0^M (1-\eta) (\partial_y^{-1} a)^2 f'' \leq 0.
\]
The term $J$ above is the leading term and it shall be bounded using integration by parts, which is justified since $g'$ vanishes as $y \to \infty$. Since $\partial_{y}^{-1} a|_{y=0}$, we obtain from \eqref{eq:g:infty} and \eqref{eq:IK:g:decreasing}  that
\begin{align*}
J &=\int_0^\infty \eta (\partial_y^{-1} a)^2 g'' \notag\\
&= - 2 \int_0^\infty \eta a (\partial_y^{-1} a) g' - \int_0^\infty \partial_y\eta (\partial_y^{-1} a)^2 g' \notag\\
&= 2 \int_0^\infty \eta a (\partial_y^{-1} a) |g'| + \int_0^\infty \partial_y\eta (\partial_y^{-1} a)^2 |g'|.
\end{align*}
Further, recalling that $\eta$ is supported on $[M,\infty)$, by appealing to \eqref{eq:IK:KEY:1} and \eqref{eq:IK:KEY:2} we get
\begin{align*}
J &\leq 2 \left( \int_0^\infty \eta (\partial_y^{-1} a)^2 g'' \right)^{1/2} \left( \int_0^\infty \eta a^2 \frac{(g')^2}{g''} \right)^{1/2} + \int_0^\infty \eta' (\partial_y^{-1} a)^2 |g'| \notag\\
&\leq 2 \sqrt{\beta} \left( \int_0^\infty \eta (\partial_y^{-1} a)^2 g'' \right)^{1/2} \left( \int_0^\infty   a^2 w \right)^{1/2} + \int_0^\infty \eta' (\partial_y^{-1} a)^2 |g'| \notag\\
&= 2 \sqrt{\beta} J^{1/2} I_2^{1/2} + \int_0^\infty \eta' (\partial_y^{-1} a)^2 |g'|.
\end{align*}
Using the inequality $2xy \leq x^2/2 +2 y^2$, we further estimate
\begin{align*} 
J \leq \frac{J}{2} + 2 \beta I_2 +   \int_0^\infty \eta' (\partial_y^{-1} a)^2 |g'|
\end{align*}
which in turn yields
\begin{align} 
J \leq 4 \beta I_2 +  2 \int_0^\infty \eta' (\partial_y^{-1} a)^2 |g'|.
\label{eq:IK:I3:2}
\end{align}
Combining \eqref{eq:IK:I3:1} and \eqref{eq:IK:I3:2}, we obtain 
\begin{align} 
I_3 
\leq 4 \beta I_2 +  2 \int_M^{Q} \eta' (\partial_y^{-1} a)^2 |g'| + \int_M^{Q}  (\partial_y^{-1} a)^2 (f''- \eta g'').
\label{eq:IK:I3:3}
\end{align}
By appealing to \eqref{eq:IK:KEY:3} 
we then arrive at the bound
\begin{align} 
I_3 \leq 4 \beta I_2
\label{eq:IK:I3}
\end{align}
which is a convenient estimate when $\beta <1$.

\subsubsection{Bound for $I_4$}
Consider 
\begin{align} 
I_4  = \int_0^\infty L[a] w  = - 2 \int_0^\infty a \, \phi \, w  + \int_0^\infty \partial_y a\, \partial_y^{-1} \phi \, w  + \int_0^\infty \partial_y \phi\, \partial_y^{-1} a\, w.
\label{eq:I4:test}
\end{align}
Using that \eqref{eq:phi:positive}--\eqref{eq:phi:increasing} hold (cf.~Lemma~\ref{lem:F}), upon integrating by parts in the second term on the far right side of \eqref{eq:I4:test}, we have that 
\begin{align*} 
I_4 
&\geq - 3 \int_0^\infty a \, \phi \, w - \int_0^\infty a\, \partial_y^{-1} \phi \, w' \notag\\
&\geq - 3 C_\kappa(1+t) \GG - \int_0^Q a \, \partial_y^{-1} \phi \, (w')_+
\end{align*}
where $(w')_+ = \max\{ w',0\}$. In the last step we used that $w$ is decreasing on $[Q,\infty)$.
Since the bounds \eqref{eq:IK:f>0}, \eqref{eq:IK:f:d1}, and \eqref{eq:phi:bounded} hold on $[0,Q]$, we arrive at
\begin{align*} 
\partial_y^{-1} \phi (w')_+ \leq y \phi(Q) (w')_+ \leq C_\kappa (1+t) \bar c_f w
\end{align*}
for $y \in [0,Q]$. We thus obtain 
\begin{align} 
I_4 \geq - (3+ \bar c_f )C_\kappa (1+t) \GG 
\label{eq:IK:I4}.
\end{align}

\subsubsection{The lower bound for the growth of the Lyapunov functional}
Combining \eqref{eq:IK:dt:G}, \eqref{eq:IK:I1}, \eqref{eq:IK:I2},  \eqref{eq:IK:I3}, and \eqref{eq:IK:I4}, with the assumption that $\beta < 1$, we arrive at 
\begin{align} 
\frac{d}{dt} \GG 
&\geq -c_{f} \GG + 2 I_2 (1-\beta) - (3+ \bar c_f )C_\kappa (1+t) \GG  \notag\\ 
&\geq - (c_f + (3+\bar c_f) C_\kappa) (1+t) \GG + \frac{2 (1-\beta)}{c_1} \GG^2\notag\\
&\geq -C_{\kappa,w} (1+t) \GG + \frac{1}{C_{\kappa,w}} \GG^2
\label{eq:dt:G}
\end{align}
for some sufficiently large positive constant $C_{\kappa,w} \geq 1$, which only depends on the choice of $\kappa$ and the weight $w$. Note that $\beta<1$  is essential here. 

\subsection{Conclusion of the proof of Theorem~\ref{thm:main}}
Therefore, if we ensure that $\GG_0 = \GG|_{t=0}$ is sufficiently large, the solution $\GG(t)$ of \eqref{eq:dt:G} blows up in finite time. Quantitatively, it is sufficient to let
\begin{align} 
\GG_0 \geq 4 C_{\kappa,w}^2.
\label{eq:G:0}
\end{align}
The condition \eqref{eq:G:0} may be achieved by a smooth initial datum. For instance we may let $a_0(y) = a(t,y)|_{y=0}$ be given by a large amplitude Gaussian bump, i.e., $a_0(y) = A \varphi(y)$, where $\varphi$ is as in \eqref{eq:IC:2} above, and $A>0$ is sufficiently large. 
In view of Remark~\ref{rem:datum}, more general classes of functions $\varphi(y)$ may be considered, including those with compact support in $y$ (cf.~\cite{CannoneLombardoSammartino01,KukavicaVicol13a}).

\section{Properties of the boundary condition lift $\phi$}
\label{sec:phi}

It is easy to verify that the function $\phi$ defined in \eqref{eq:phi:particular}, i.e.,
\begin{align} 
\phi(t,y) &=   \kappa \erf\left( \frac{y}{\sqrt{4(t+1)}} \right) \notag\\
&\quad + \kappa^2 t \left( \frac{y^2}{2t} \left( \erf\left( \frac{y}{\sqrt{4 t}} \right) - 1 \right) + \left( \erf\left( \frac{y}{\sqrt{4 t}} \right) + \frac{y}{\sqrt{\pi t}} \exp\left( - \frac{y^2}{4 t}\right) \right) \right) \notag \\
&=   \kappa \erf\left( \frac{y}{\sqrt{4(t+1)}} \right) \notag\\
&\quad + \kappa^2 t \left( 2 z(t,y)^2 \left( \erf\left( z(t,y) \right) - 1 \right) + \left( \erf\left( z(t,y) \right) + \frac{2 z(t,y) }{\sqrt{\pi }} \exp\left( - z(t,y)^2\right) \right) \right)
\label{eq:phi:def:2}
\end{align}
obeys the non homogenous heat equation \eqref{eq:phi:1}--\eqref{eq:phi:3}, with initial value $\phi_0(y) = \kappa \erf(y/2)$, where $z(t,y) = y/\sqrt{4 t}$ is the heat self-similar variable.

\subsection{Proof of Lemma~\ref{lem:F}}

\begin{proof}[Proof of \eqref{eq:phi:increasing}]
First note that for $t>0$ and $y>0$, the function $\partial_y \phi$ obeys the heat equation, i.e., $(\partial_t -\partial_{yy})(\partial_y  \phi) =0$. Using the exact formula \eqref{eq:phi:particular} we obtain the initial and boundary values for the quantity $\partial_y \phi$. 
Taking the $y$ derivative of $\phi$ gives
  \begin{align*}
  \partial_y \phi 
  =
  \frac{\kappa}{\sqrt{\pi\tt}}\exp{\left(-\frac{y^2}{4\tt}\right)}
  +&
  \kappa^2\bigg( y\left(\erf \left( \frac{y}{\sqrt{4t}} \right)-1\right)
  +
  \frac{y^2}{\sqrt{4\pi t}}\exp\left( - \frac{y^2}{4 t}\right)\nonumber\\
  &
  \qquad + 
  t\left(
   \frac{2}{\sqrt{\pi t}}\exp\left( - \frac{y^2}{4 t}\right)
   - \frac{y^2}{2t\sqrt{\pi t}}\exp\left( - \frac{y^2}{4 t}\right)\right)\bigg)
  \end{align*}
  where $\tt = t+1$.
 Sending $y \to 0$ and $y \to \infty$ we obtain
\begin{align*}
\partial_y \phi|_{y=0} &= \frac{\kappa}{\sqrt{\pi\tt}} + \frac{2 \kappa^2 \sqrt{t}}{\sqrt{\pi}} >0 \\
\partial_y  \phi|_{y=\infty} &= 0
\end{align*}
for all $t>0$.
Taking the limit $t\rightarrow0$, we arrive at 
  \begin{align*}
\partial_y  \phi|_{t=0}=\frac{\kappa}{\sqrt{\pi}}\exp{\left(-\frac{y^2}{4}\right)}\ge0.
  \end{align*}
The fact  $\partial_y \phi(t,y) \geq 0$ for $t,y\geq 0$ now follows from the parabolic maximum principle. 
For the sake of completeness we repeat this classical argument. We  consider the nonnegative $C^2$ function 
  \begin{align*}
  \ff(x)=
  \begin{cases}
  x^4, &\quad x\le 0\\
  0, &\quad x\ge0.
  \end{cases}
  \end{align*}
Taking the $t$ derivative of the quantity $\int_0^{\infty}\ff(\partial_y \phi(t,y) ) dy$, upon integrating by parts in $y$ and using the boundary conditions for $\partial_y \phi$ we arrive at
  \begin{align*}
  \partial_t\int_0^{\infty}\ff(\partial_y \phi)\,dy 
  &= 
  \int_0^{\infty}\ff'(\partial_y \phi) \partial_t \partial_y \phi \,dy 
  =
  \int_0^{\infty}\ff'(\partial_y \phi)\partial^3_y\phi\,dy\notag\\
  &=
  -\int_0^{\infty}\ff''(\partial_y \phi)(\partial^2_y\phi)^2\,dy - \ff'(\partial_y \phi|_{y=0}) \partial_{yy} \phi|_{y=0}\notag\\
  &=
  -\int_0^{\infty}\ff''(\partial_y \phi)(\partial^2_y\phi)^2\,dy \le0,
  \end{align*}
from where we deduce that $\int_0^{\infty}\ff(\partial_y \phi(t,y))\,dy=0$ for $t\ge0$ since $\int_0^{\infty}\ff(\partial_y \phi_0(y))\,dy =0.$
Therefore, we  obtain  $\partial_y \phi(t,y)\ge0$, concluding the proof.
\end{proof}

\begin{proof}[Proof of \eqref{eq:phi:positive}]
Since $\phi(t, 0)=0$, for all $t>0$, the non-negativity of $\phi$ follows from the fundamental theorem of calculus and the above established monotonicity property $\partial_y \phi \geq 0$.
\end{proof}

\begin{proof}[Proof of \eqref{eq:phi:bounded}]
The proof follows from \eqref{eq:phi:def:2} since   
\begin{align*}
  2 z^2 \left( \erf\left( z  \right) - 1\right) + \erf(z) + \frac{2}{\sqrt{\pi}} z \exp(-z^2) \leq C_0
  \end{align*}
for all $z\geq 0$, for some universal constant $C_0 > 0$. We may then take $C_\kappa = \max\{ \kappa,C_0 \kappa^2 \}$.
\end{proof}

\begin{proof}[Proof of \eqref{eq:phi:concave}]
Taking the second derivative of $\phi$ defined in formula~\eqref{eq:phi:particular}, we arrive at
  \begin{align}
  \phiyy 
  =
  &-\frac{\kappa y}{\sqrt{4 \pi\tt^3}}\exp{\left(-\frac{y^2}{4\tt}\right)}
   - \kappa^2 t\bigg( \frac{2y}{\sqrt{\pi t^3}}\exp\left( - \frac{y^2}{4 t}\right) 
    -
   \frac{y^3}{4\sqrt{\pi t^5}}\exp\left( - \frac{y^2}{4 t}\right)\bigg)
    \nonumber\\
  &+ \kappa^2\bigg( \erf\frac{y}{\sqrt{4t}}-1 
  +
  \frac{2y}{\sqrt{\pi t}}\exp\left( - \frac{y^2}{4 t}\right) 
  -
  \frac{y^3}{4\sqrt{\pi t^3}}\exp\left( - \frac{y^2}{4 t}\right) \bigg).
     \label{eq:phi:yy}
  \end{align}
From this expression we get the initial and boundary values for $\phiyy$ as
  \begin{align*}
  \phiyy|_{y=0}&=-\kappa^2\le0,  \qquad \mbox{for }t>0\\
  \phiyy|_{y\rightarrow\infty}&=0,  \qquad \mbox{for }t>0,\\
  \phiyy|_{t=0}&=-\frac{\kappa y}{2\sqrt{\pi}}\exp{\left(-\frac{y^2}{4}\right)}\le0, \qquad \mbox{for }y>0.
  \end{align*}
  An argument similar to the one above shows that by the parabolic maximum principle we have $\partial_{yy} \phi (t,y)\leq 0$ for all $t,y\geq 0$.
\end{proof}

\section{Construction of a weight function $w$ for the Lyapunov functional}
\label{sec:w}
We fix 
$Q = 1$
and let $r>1$ be a free parameter, to be chosen below. In terms of this $r$ we shall pick $1/2 <M = M(r)<1$, $ 0 < \beta = \beta(r) < 1 $, and $B = B(r)>1$ so that the conditions \eqref{eq:f:0}--\eqref{eq:IK:KEY:3} hold.

Define the function $f$ by
\begin{align}
f(y) = \frac{2B+r}{B^r} y  - \frac{B+r}{B^r} y^2
\label{eq:f:def}
\end{align}
Therefore, $f(0) = 0$	 and
\begin{align*}
f'(y) = \frac{2B+r}{B^r}  - \frac{2 (B+r)}{B^r} y 
\end{align*}
with
\begin{align*}
f''(y) = - \frac{2(B+r)}{B^r}.
\end{align*}
It thus follows that $f \in W^{2,\infty}([0,1])$, $f \geq 0$, and $f'' \leq 0$ on $(0,1)$, so that \eqref{eq:f:0}--\eqref{eq:IK:f:d2<0} hold. 
Moreover, \eqref{eq:IK:f:d1} holds with $\bar c_f = 1$.
Note however that \eqref{eq:IK:f:d2:lower} does not hold in a neighborhood of the origin, since $f(0)=0$. Instead, the function $f$ defined in \eqref{eq:f:def} needs to be modified in a small neighborhood near the origin so that it is linear there (see Remark~\ref{rem:f:near:0} below).

Next, we define 
\begin{align*} 
g(y) = \frac{B}{(y+B-1)^r}
\end{align*}
set initially for all $y \geq 1$, but which is a well-defined function on $y \geq M$ as long as $M+B>1$. Note that $r>1$ implies that \eqref{eq:IK:w:L1} holds.
We have that
\begin{align*} 
g'(y) = - \frac{r B}{(y+B-1)^{r+1}} < 0 
\end{align*}
and
\begin{align*}
g''(y) =  \frac{r (r+1) B}{(y+B-1)^{r+2}} >0.
\end{align*}
Thus, the properties \eqref{eq:g:infty}--\eqref{eq:g:convex} hold for this function $g$.
Moreover,
\begin{align*} 
\frac{g'(y)^2}{g(y) g''(y)} = \frac{r}{r+1} < 1
\end{align*}
for all $y>0$.
This verifies that condition \eqref{eq:IK:KEY:1} holds for any $\beta \in [r/(r+1),1)$.

Note that at $y=Q=1$ we have
\begin{align*} 
f(1) = B^{1-r} = g(1)
\end{align*}
and
\begin{align*} 
f'(1) = - r B^{-r} = g'(1)
\end{align*}
and thus $f$ and $g$ may be glued together at $Q=1$ to yield a $W^{2,\infty}$ function on $\RR$.

In order to assure that \eqref{eq:IK:KEY:2} holds, it is sufficient to verify that
\begin{align} 
\frac{r}{r+1} \psi\left(\frac{y-M}{1-M}\right) g(y) \leq \beta f(y)
\label{eq:IK:CHECK:2}
\end{align}
holds for all $y \in [M,1]$, where $r/(r+1) < \beta < 1$ is arbitrary. The condition \eqref{eq:IK:CHECK:2} holds automatically 
with
\begin{align*} 
\beta = \frac{2r+1}{2r+2}
\end{align*}
if we ensure that 
\begin{align} 
\sup_{y \in [M,1]} \frac{g(y)}{f(y)} \leq \frac{2r+1}{2r}.
\label{eq:IK:CHECK:21}
\end{align}
In view of the continuity of the above functions, \eqref{eq:IK:CHECK:21} holds if we choose $M \in (0,1)$ sufficiently close to $1$. We need though to be more precise on this choice of $M$. Indeed, \eqref{eq:IK:CHECK:21} holds for $y \in [M,1]$ if we impose that 
\begin{align*} 
\frac{B^{r+1}}{(y+B-1)^r} &\leq \frac{2r+1}{2r} y \left( (2B+r) - (B+r) y \right)
\end{align*}
which is a consequence of
\begin{align*}
\frac{B^{r+1}}{(M+B-1)^r} &\leq \frac{2r+1}{2r} M B.
\end{align*}
Assuming that $1-M\leq 1/(4r+2)$,  the above follows from 
\begin{align*}
\frac{1}{\left(1- (1-M)/B\right)^r} \leq \frac{4r+1}{4r} 
\end{align*}
which holds provided that
\begin{align*}
 \left( \frac{4r}{4r+1}\right)^{1/r} \leq 1- \frac{1-M}{B}.
\end{align*}
The last condition may be written as 
\begin{align}
1-M \leq B  \left( 1 - \left( \frac{4r}{4r+1}\right)^{1/r} \right).
\label{eq:r:cond}
\end{align}
Therefore, \eqref{eq:r:cond} holds if we choose
\begin{align} 
1-M = \min\left\{ \frac{1}{4r+2} , B  \left( 1 - \left( \frac{4r}{4r+1}\right)^{1/r} \right) \right\} = \frac{1}{4r+2}
\label{eq:M:def}
\end{align}
as long as
\begin{align} 
B \geq \frac{1}{(4r+2)\left(1 - \left( (4r)/(4r+1)\right)^{1/r}\right)} \label{eq:B:cond:1}.
\end{align}
This ensures the validity of the  condition \eqref{eq:IK:CHECK:2}, and thus also \eqref{eq:IK:KEY:2} holds.

We finally verify that \eqref{eq:IK:KEY:3} holds, or equivalently 
\begin{align} 
\frac{2}{1-M}\psi'\left(\frac{y-M}{1-M}\right) \frac{r B}{(y+B-1)^{r+1}} \leq \psi\left(\frac{y-M}{1-M}\right) \frac{r (r+1) B}{(y+B-1)^{r+2}} + \frac{2 (B+r)}{B^r}.
\label{eq:IK:CHECK:3}
\end{align}
Since $|\psi'| \leq 2$ and $\phi \geq 0$, the above condition holds on $[M,1]$ once we ensure that 
\begin{align*} 
\frac{4}{1-M} \frac{r B}{(y+B-1)^{r+1}} \leq \frac{2 (B+r)}{B^r}
\end{align*}
which is implied by
\begin{align*}
\frac{B^{r+1}}{(M+B-1)^{r+1}} \leq \frac{(1-M)(B+r)}{2r}.
\end{align*}
The above condition holds if we take $B$ sufficiently large, depending only on $r$. More precisely, since $M$ obeys \eqref{eq:M:def}, letting $B$ obey \eqref{eq:B:cond:1} and also 
\begin{align}
B \geq 2r (4r+2) \left(\frac{4r+1}{4r}\right)^{(r+1)/r},
\label{eq:B:cond:2}
\end{align}
we complete the proof of \eqref{eq:IK:CHECK:3}.

In summary, the conditions \eqref{eq:f:0}--\eqref{eq:IK:KEY:3}, except for \eqref{eq:IK:f:d2:lower}, are obeyed once we set
\begin{align*} 
r&=2 \\
\beta &= \frac{5}{6} < 1 \\
M &= \frac{9}{10} <  Q = 1\\
B &=50.
\end{align*}

\subsection{Condition \eqref{eq:IK:f:d2:lower}}
 \label{rem:f:near:0}
In order to  ensure that \eqref{eq:IK:f:d2:lower} holds, we need to tweak the functions $f$ and $g$ defined above.
Let $0 < \eps < 1/2$ be a small parameter, to be determined.
We then have 
\begin{align*}
 f(\eps) &= \frac{2B+r}{B^r} \eps - \frac{B+r}{B^r} \eps^2 > 0\\
 f'(\eps) &= \frac{2B+r}{B^r}  - \frac{2(B+r)}{B^r} \eps >0.
\end{align*}
Therefore, we can extend $f(y)$ by the linear function 
\begin{align*}
h_\eps(y) &= f'(\eps) (y-\eps) + f(\eps) \notag\\
&= \left(\frac{2B+r}{B^r}  - \frac{2(B+r)}{B^r} \eps\right) (y-\eps) + \left( \frac{2B+r}{B^r} \eps - \frac{B+r}{B^r} \eps^2 \right)
\end{align*} 
on the interval $[- y_\eps,\eps]$, where 
\begin{align*}
y_\eps= \frac{f(\eps)}{f'(\eps)} - \eps \geq 0.
\end{align*}
Note that $h_\eps(-y_\eps) = 0$, $h_\eps(\eps) = f(\eps)$, and $h_\eps'(\eps) = f'(\eps)$. Therefore, glueing $h_\eps$ with $f$ at $y=\eps$, and then $f$ with $g$ at $y=Q=1$, yields a $W^{2,\infty}$ function 
\begin{align*}
w_\eps(y)=
\begin{cases}
h_\eps(y), &y \in [-y_\eps,\eps] \\
f(y), &y \in [\eps,1] \\
g(y), & y \geq 1
\end{cases}
\end{align*}
which obeys all the properties \eqref{eq:f:0}--\eqref{eq:IK:KEY:3}, but on the interval $[-y_\eps,\infty)$ instead of $[0,\infty)$. Here we used that $\eps$ is sufficiently small so that $\eps \leq M/4$. Also, \eqref{eq:IK:f:d2:lower} holds trivially on $[-y_\eps,\eps)$ since on this interval $h_\eps''(y) = 0 \geq - h_\eps(y)$. Then in view of the continuity of $f''$ and $f$, and the fact that $f$ only vanishes at $y=0$, on the compact $[\eps,1]$ we have that $|f''(y)|/f(y) \leq c_f$ for some constant $c_f \geq 1$. 

Therefore, to complete the construction of the weight function $w$,  let 
\begin{align*}
 w(y) = w_\eps(y + y_\eps)
\end{align*}
for a fixed $0< \eps \ll1$, where the values of $Q$ and $M$ are themselves shifted by $y_\eps$. In view of the smoothness of all parameters on $\eps$ this is possible without affecting \eqref{eq:IK:KEY:1}--\eqref{eq:IK:KEY:2}, upon slightly increasing the value of $\beta$ (so that it remains $<1$).

\section*{Acknowledgments}
IK and FW were supported in part by the
NSF grant DMS-1311943, 
while VV was supported in part by the NSF grant DMS-1514771 and by an
Alfred P.~Sloan Research Fellowship.


\end{document}